\documentclass[12pt,reqno,a4paper]{amsart}%
\usepackage{amsfonts}
\usepackage{amsmath}
\usepackage{amssymb}
\usepackage{graphicx}%
\newtheorem{theorem}{Theorem}[section]

\theoremstyle{definition}
\newtheorem{definition}[theorem]{Definition}
\newtheorem{remark}[theorem]{Remark}
\numberwithin{equation}{section} 
\begin{document}
\title[A Comprehensive Subclass of Analytic and Bi-Univalent Functions...]{A Comprehensive Subclass of Analytic and Bi-Univalent Functions Associated with Subordination}
\author{A. A. Amourah}
\address{A. A. Amourah: Department of Mathematics, Faculty of Science and Technology,
Irbid National University, Irbid, Jordan.}
\email{alaammour@yahoo.com}

\author{Mohamed Illafe}
\address{Mohamed Illafe: School of Engineering, Math, $\&$ Technology, Navajo Technical University, Crownpoint, NM 87313, USA.}
\email{millafe@navajotech.edu}

\begin{abstract}
In the present paper, we define a new general subclass of bi-univalent
functions involving a differential operator in the open unit disk $\mathbb{U}%
$. For this purpose, we use the Faber polynomial expansions. Several
connections to some of the earlier known results are also pointed out.

\end{abstract}
\maketitle

\section{Introduction}

Let $\mathcal{A}$ denote the class of all analytic functions $f$ defined in
the open unit disk $\mathbb{U}=\{z\in\mathbb{C}:\left\vert z\right\vert <1\}$
and normalized by the conditions $f(0)=0$ and $f^{\prime}(0)=1$. Thus each
$f\in\mathcal{A}$ has a Taylor-Maclaurin series expansion of the form:%
\begin{equation}
f(z)=z+\sum\limits_{n=2}^{\infty}a_{n}z^{n},\ \ (z\in\mathbb{U}).\text{
\ \ \ \ } \label{ieq1}%
\end{equation}

Further, let $\mathcal{S}$ denote the class of all functions $f\in\mathcal{A}$
which are univalent in $\mathbb{U}$ (for details, see \cite{Duren}; see also
some of the recent investigations \cite{C6,C4,C1,C5,C2}). And let $\mathcal{C}$ be the class of functions $\Phi(z)=1+\sum\limits_{n=1}^{\infty
}\Phi_{n}z^{n}$ that are analytic in $\mathbb{U}$ and satisfy the condition
$\operatorname{Re}(\Phi(z))>0$ in $\mathbb{U}$. By the Caratheodory's lemma
(see \cite{Duren}) we have $\left\vert \Phi(z)\right\vert \leq2.$ \\

Let the functions $f,$ $g$ be analytic in $\mathbb{U}$. If there exists a
Schwarz function $\varpi,$ is analytic in $\mathbb{U}$ under the conditions%
\[
\varpi(0)=0,\text{ }\left\vert \varpi(z)\right\vert \leq1,
\]
such that%
\[
f(z)=g(\varpi(z)),\text{ }z\in\mathbb{U},
\]
then, the function $f$ is subordinate to $g$ in $\mathbb{U}$, and
we write $f(z)\prec g(z).$ \\

The Koebe one-quarter theorem (for details, (see \cite{Duren}), we
know that the image of $\mathbb{U}$ under every function $f\in\mathcal{A}$
contains a disk of radius $\frac{1}{4}$. According to this, every function
$f\in\mathcal{A}$ has an inverse map $f^{-1}$ that satisfies the following conditions:
\begin{center}%
\[
f^{-1}(f(z))=z\ \ \ (z\in\mathbb{U}),
\]
\end{center}
and
\begin{center}%
\[
f\left(  f^{-1}(w)\right)  =w\ \ \ \left(  |w|<r_{0}(f);r_{0}(f)\geq\frac
{1}{4}\right)  .
\]
\end{center}

In fact, the inverse function is given by%
\begin{equation}
g(w)=f^{-1}(w)=w-a_{2}w^{2}+(2a_{2}^{2}-a_{3})w^{3}-(5a_{2}^{3}-5a_{2}%
a_{3}+a_{4})w^{4}+\cdots. \label{ieq2}%
\end{equation}

A function $f\in\mathcal{A}$ is said to be bi-univalent in $\mathbb{U}$ if
both $f(z)$ and $f^{-1}(z)$ are univalent in $\mathbb{U}$. Let $\Sigma$ denote
the class of bi-univalent functions in $\mathbb{U}$ given by (\ref{ieq1}).
Examples of functions in the class $\Sigma$ are%
\[
\frac{z}{1-z},\ -\log(1-z),\ \frac{1}{2}\log\left(  \frac{1+z}{1-z}\right)
,\cdots.
\]

It is worth noting that the familiar Koebe function is not a member of
$\Sigma$, since it maps the unit disk $\mathbb{U}$ univalently onto the entire
complex plane except the part of the negative real axis from $-1/4$ to
$-\infty$. Thus, clearly, the image of the domain does not contain the unit
disk $\mathbb{U}$. For a brief history and some intriguing examples of
functions and characterization of the class $\Sigma$, see Srivastava et al.
\cite{C27}, Yousef et al. \cite{C3,class,cc}, and Frasin and Aouf \cite{C28}. \\

In 1967, Lewin \cite{C21} investigated the bi-univalent function class
$\Sigma$ and showed that $|a_{2}|<1.51$. Subsequently, Brannan and Clunie
\cite{C22} conjectured that $|a_{2}|\leq\sqrt{2}.$ On the other hand,
Netanyahu \cite{C23} showed that $\underset{f\in\Sigma}{\max}$ $|a_{2}%
|=\frac{4}{3}.$ The best known estimate for functions in $\Sigma$ has been
obtained in 1984 by Tan \cite{C24}, that is, $|a_{2}|<1.485$. The coefficient
estimate problem for each of the following Taylor-Maclaurin coefficients
$|a_{n}|$ $(n\in\mathbb{N}\backslash\{1,2\})$ for each $f\in\Sigma$ given by
(\ref{ieq1}) is presumably still an open problem. \\

The Faber polynomials introduced by Faber \cite{12} play an important role in
various areas of mathematical sciences, especially in geometric function
theory. The recent publications \cite{15} and \cite{17} applying the Faber
polynomial expansions to meromorphic bi-univalent functions motivated us to
apply this technique to classes of analytic bi-univalent functions. In the
literature, there are only a few works determining the general coefficient
bounds $|a_{n}|$ for the analytic bi-univalent functions given by (\ref{ieq1})
using Faber polynomial expansions (see for example, \cite{16,20,19}). Hamidi
and Jahangiri \cite{16} considered the class of analytic bi-close-to-convex
functions. Jahangiri and Hamidi \cite{19} considered the class defined by
Frasin and Aouf \cite{C28}, and Jahangiri et al. \cite{20} considered the
class of analytic bi-univalent functions with positive real-part derivatives.
\newpage

\section{The class $\mathfrak{B}_{\Sigma}(\mu
,\lambda,\Phi,\xi)$}

Yousef et al. \cite{class} have introduced and studied the following subclass of analytic bi-univalent functions:

\begin{definition}
For $\lambda\geq1,$ $\mu\geq0,$ $\delta\geq0$
and $0\leq\alpha<1$, a function $f\in\Sigma$ given by (\ref{ieq1}) is said to
be in the class $\mathfrak{B}_{\Sigma}^{\mu}(\alpha,\lambda,\delta)$ if the
following conditions hold for all $z,w\in\mathbb{U}$:
\begin{equation}
\mbox{Re}\left(  (1-\lambda)\left(  \frac{f(z)}{z}\right)  ^{\mu}+\lambda
f^{\prime}(z)\left(  \frac{f(z)}{z}\right)  ^{\mu-1}+\xi\delta zf^{\prime
\prime}(z)\right)  >\alpha%
\end{equation}
and%
\begin{equation}
\mbox{Re}\left(  (1-\lambda)\left(  \frac{g(w)}{w}\right)  ^{\mu}+\lambda
g^{\prime}(w)\left(  \frac{g(w)}{w}\right)  ^{\mu-1}+\xi\delta wg^{\prime
\prime}(w)\right)  >\alpha, %
\end{equation}
where the function $g(w)=f^{-1}(w)$ is defined by (\ref{ieq2}) and $\xi
=\frac{2\lambda+\mu}{2\lambda+1}$.
\end{definition}

Using the Faber polynomial expansion of functions $f\in\mathcal{A}$ of the
form (\ref{ieq1}), the coefficients of its inverse map $g=f^{-1}$ may be
expressed as in \cite{z1}:%
\begin{equation}
g(w)=f^{-1}(w)=w+\sum\limits_{n=2}^{\infty}\frac{1}{n}K_{n-1}^{-n}\left(
a_{2},a_{3},...\right)  w^{n}, \label{ieq3}%
\end{equation}
where%
\begin{align}
K_{n-1}^{-n}  &  =\frac{\left(  -n\right)  !}{\left(  -2n+1\right)  !\left(
n-1\right)  !}a_{2}^{n-1}+\frac{\left(  -n\right)  !}{\left(  2\left(
-n+1\right)  \right)  !\left(  n-3\right)  !}a_{2}^{n-3}a_{3}+\frac{\left(
-n\right)  !}{\left(  -2n+3\right)  !\left(  n-4\right)  !}a_{2}^{n-4}%
a_{4}\label{ieq4}\\
&  +\frac{\left(  -n\right)  !}{\left(  2\left(  -n+2\right)  \right)
!\left(  n-5\right)  !}a_{2}^{n-5}\left[  a_{5}+\left(  -n+2\right)  a_{3}%
^{2}\right]  +\frac{\left(  -n\right)  !}{\left(  -2n+5\right)  !\left(
n-6\right)  !}a_{2}^{n-6}\nonumber\\
&  \left[  a_{6}+\left(  -2n+5\right)  a_{3}a_{4}\right]  +\sum\limits_{j\geq
7}a_{2}^{n-j}V_{j},\nonumber
\end{align}
such that $V_{j}$ with $7\leq j\leq n$ is a homogeneous polynomial in the
variables $a_{2},a_{3},...,a_{n}$ \cite{z3}. \\

In particular, the first three terms of $K_{n-1}^{-n}$ are%
\begin{equation}
K_{1}^{-2}=-2a_{2},\text{ }K_{2}^{-3}=3\left(  2a_{2}^{2}-a_{3}\right)
,\text{ }K_{3}^{-4}=-4\left(  5a_{2}^{3}-5a_{2}a_{3}+a_{4}\right)  .
\label{ieq5}%
\end{equation}
\bigskip

In general, for any $p\in%
\mathbb{N}
$ $:=\{1,2,3,...\}$, an expansion of $K_{n}^{p}$ is as in \cite{z1},%
\begin{equation}
K_{n}^{p}=pa_{n}+\frac{p(p-1)}{2}D_{n}^{2}+\frac{p!}{\left(  p-3\right)
!3!}D_{n}^{3}+\cdots+\frac{p!}{\left(  p-n\right)  !n!}D_{n}^{n}, \label{ieq6}%
\end{equation}
where $D_{n}^{p}=D_{n}^{p}\left(  a_{2},a_{3},...\right)  ,$ and by
\cite{z29}, $D_{n}^{m}\left(  a_{1},a_{2},...,a_{n}\right)  =$ $\sum
\limits_{n=1}^{\infty}\frac{m!}{i_{1}!...i_{n}!}a_{1}^{i_{1}}...a_{n}^{i_{n}}$
while $a_{1}=1,$ and the sum is taken over all non-negative integers
$i_{1},...,i_{n}$ satisfying $i_{1}+i_{2}+\cdots+i_{n}=m,$ $i_{1}%
+2i_{2}+\cdots+ni_{n}=n,$ it is clear that $D_{n}^{m}\left(  a_{1}%
,a_{2},...,a_{n}\right)  =a_{1}^{n}.$ \\

Now, we are ready to establish a new subclass of analytic and bi-univalent functions based
on subordination.

\begin{definition}
\label{def221} For $\lambda\geq1,$ $\mu\geq0,$ and $\delta\geq0$, A function
$f\in\Sigma$ is said to be in the class $\mathfrak{B}_{\Sigma}(\mu
,\lambda,\Phi,\xi),$ if the following subordinations are satisfied:%
\begin{equation}
(1-\lambda)\left(  \frac{f(z)}{z}\right)  ^{\mu}+\lambda f^{\prime}(z)\left(
\frac{f(z)}{z}\right)  ^{\mu-1}+\xi\delta zf^{\prime\prime}(z)\prec\Phi(z)
\label{def223}%
\end{equation}
and%
\begin{equation}
(1-\lambda)\left(  \frac{g(w)}{w}\right)  ^{\mu}+\lambda g^{\prime}(w)\left(
\frac{g(w)}{w}\right)  ^{\mu-1}+\xi\delta wg^{\prime\prime}(w)\prec\Phi(w)
\label{def224}%
\end{equation}
where the function $g(w)=f^{-1}(w)$ is defined by (\ref{ieq2}) and $\xi
=\frac{2\lambda+\mu}{2\lambda+1}$.
\end{definition}

\section{Coefficient bounds for the function class $\mathfrak{B}_{\Sigma}%
(\mu,\lambda,\Phi,\xi)$}

\begin{theorem}
\label{thm1} For $\lambda\geq1,$ $\mu\geq0,$and $\delta\geq0$, let the
function $f\in$ $\mathfrak{B}_{\Sigma}(\mu,\lambda,\Phi,\xi)$ be given by
(\ref{ieq1}). Then
\[
\left\vert a_{2}\right\vert \leq\min\left\{  \frac{2}{\mu+\lambda+2\xi\delta
},\sqrt{\frac{8}{\left(  \mu+2\lambda+6\xi\delta\right)  \left(  \mu+1\right)
}}\right\}  \text{ }%
\]
and%
\[
\left\vert a_{3}\right\vert \leq\min\left\{
\begin{array}
[c]{c}%
\frac{1}{\left(  1+\frac{6\delta}{2\lambda+1}\right)  }\left[  \frac
{4}{\left(  \mu+\lambda+2\xi\delta\right)  ^{2}}+\frac{2}{\left(  \mu
+2\lambda+6\xi\delta\right)  }\right]  ,\\
\frac{1}{\left(  1+\frac{6\delta}{2\lambda+1}\right)  }\left[  \frac
{8}{\left(  \mu+2\lambda+6\xi\delta\right)  \left(  \mu+1\right)  }+\frac
{2}{\left(  \mu+2\lambda+6\xi\delta\right)  }\right]
\end{array}
\right\}.
\]

\end{theorem}

\begin{proof}
Let $\mathfrak{B}_{\Sigma}(\mu,\lambda,\Phi,\xi)$.The inequalities
(\ref{def223}) and (\ref{def224}) imply the existence of two positive real
part functions%
\[
\varpi(z)=1+\sum\limits_{n=1}^{\infty}t_{n}z^{n}%
\]
and%
\[
\varphi(w)=1+\sum\limits_{n=1}^{\infty}s_{n}z^{n}%
\]
where $\operatorname{Re}\left(  \varpi(z)\right)  >0$ and $\operatorname{Re}%
\left(  \varphi(w)\right)  >0$ in $\mathcal{C}$ so that
\begin{equation}
(1-\lambda)\left(  \frac{f(z)}{z}\right)  ^{\mu}+\lambda f^{\prime}(z)\left(
\frac{f(z)}{z}\right)  ^{\mu-1}+\xi\delta zf^{\prime\prime}(z)=\Phi\left(
\varpi(z)\right)  , \label{ieq9}%
\end{equation}%
\begin{equation}
(1-\lambda)\left(  \frac{g(w)}{w}\right)  ^{\mu}+\lambda g^{\prime}(w)\left(
\frac{g(w)}{w}\right)  ^{\mu-1}+\xi\delta wg^{\prime\prime}(w)=\Phi\left(
\varphi(w)\right)  . \label{ieq10}%
\end{equation}

It follows from (\ref{ieq9}) and (\ref{ieq10}) that
\begin{equation}
\left(  \mu+\lambda+2\xi\delta\right)  a_{2}=\Phi_{1}t_{1} \label{ieq11}%
\end{equation}%
\begin{equation}
\left(  \mu+2\lambda+6\xi\delta\right)  \left[  \frac{\mu-1}{2}a_{2}%
^{2}+\left(  1+\frac{6\delta}{2\lambda+1}\right)  a_{3}\right]  =\Phi_{1}%
t_{2}+\Phi_{2}t_{1}^{2}, \label{ieq12}%
\end{equation}
and%
\begin{equation}
-\left(  \mu+\lambda\right)  a_{2}=\Phi_{1}s_{1}, \label{ieq13}%
\end{equation}%
\begin{equation}
\left(  \mu+2\lambda+6\xi\delta\right)  \left[  \frac{\mu+3}{2}a_{2}%
^{2}-\left(  1+\frac{6\delta}{2\lambda+1}\right)  a_{3}\right]  =\Phi_{1}%
s_{2}+\Phi_{2}s_{1}^{2}. \label{ieq14}%
\end{equation}

From (\ref{ieq11}) and (\ref{ieq13}), we find%
\begin{equation}
\left\vert a_{2}\right\vert \leq\frac{\left\vert \Phi_{1}t_{1}\right\vert
}{\mu+\lambda+2\xi\delta}=\frac{\left\vert \Phi_{1}s_{1}\right\vert }%
{\mu+\lambda+2\xi\delta}\leq\frac{2}{\mu+\lambda+2\xi\delta}. \label{ieq155}%
\end{equation}

From (\ref{ieq12}) and (\ref{ieq14}), we get%
\[
\left(  \mu+2\lambda+6\xi\delta\right)  \left(  \mu+1\right)  a_{2}^{2}%
=\Phi_{1}\left(  t_{2}+s_{2}\right)  +\Phi_{2}\left(  t_{1}^{2}+s_{1}%
^{2}\right)
\]
or, equivalently%
\begin{equation}
\left\vert a_{2}\right\vert \leq\sqrt{\frac{8}{\left(  \mu+2\lambda+6\xi
\delta\right)  \left(  \mu+1\right)  }}. \label{ieq166}%
\end{equation}

Next, in order to find the bound on the coefficient $\left\vert a_{3}%
\right\vert $, we subtract (\ref{ieq14}) from (\ref{ieq12}). We thus get%
\begin{equation}
2\left(  \mu+2\lambda+6\xi\delta\right)  \left[  \left(  1+\frac{6\delta
}{2\lambda+1}\right)  a_{3}-a_{2}^{2}\right]  =\Phi_{1}\left(  t_{2}%
-s_{2}\right)  +\Phi_{2}\left(  t_{1}^{2}-s_{1}^{2}\right)  \label{ieq177}%
\end{equation}
or%
\begin{align}
\left\vert a_{3}\right\vert  &  \leq\frac{1}{\left(  1+\frac{6\delta}%
{2\lambda+1}\right)  }\left[  \left\vert a_{2}\right\vert ^{2}+\frac
{\left\vert \Phi_{1}\left(  t_{2}-s_{2}\right)  \right\vert }{2\left(
\mu+2\lambda+6\xi\delta\right)  }\right] \label{ieq188}\\
&  =\frac{1}{\left(  1+\frac{6\delta}{2\lambda+1}\right)  }\left[  \left\vert
a_{2}\right\vert ^{2}+\frac{2}{\left(  \mu+2\lambda+6\xi\delta\right)
}\right]  .\nonumber
\end{align}

Upon substituting the value of $a_{2}^{2}$ from (\ref{ieq155}) and
(\ref{ieq166}) into (\ref{ieq188}), it follows that%
\[
\left\vert a_{3}\right\vert \leq\frac{1}{\left(  1+\frac{6\delta}{2\lambda
+1}\right)  }\left[  \frac{4}{\left(  \mu+\lambda+2\xi\delta\right)  ^{2}%
}+\frac{2}{\left(  \mu+2\lambda+6\xi\delta\right)  }\right]
\]
and%
\[
\left\vert a_{3}\right\vert \leq\frac{1}{\left(  1+\frac{6\delta}{2\lambda
+1}\right)  }\left[  \frac{8}{\left(  \mu+2\lambda+6\xi\delta\right)  \left(
\mu+1\right)  }+\frac{2}{\left(  \mu+2\lambda+6\xi\delta\right)  }\right]
\]

Which completes the proof of \ref{thm1}.
\end{proof}

\begin{theorem}
\label{thm2} Let $\mathfrak{B}_{\Sigma}(\mu,\lambda,\Phi,\xi).$ If $a_{m}=0$
with $2\leq m\leq n-1,$ then
\begin{equation}
\hspace{-1in}\left\vert a_{n}\right\vert \leq\frac{2}{\mu+\left(  n-1\right)
\lambda+n\left(  n-1\right)  \xi\delta}\text{ }(n\geq4). \label{ieq15}%
\end{equation}

\end{theorem}

\begin{proof}
By using the Faber polynomial expansion of functions $f\in\mathcal{A}$ of the
form (\ref{ieq1}) and its inverse map $g=f^{-1}$, we can write%
\begin{equation}
(1-\lambda)\left(  \frac{f(z)}{z}\right)  ^{\mu}+\lambda f^{\prime}(z)\left(
\frac{f(z)}{z}\right)  ^{\mu-1}+\xi\delta zf^{\prime\prime}(z)=1+\sum
\limits_{n=2}^{\infty}F_{n-1}\left(  a_{2},a_{3},...,a_{n}\right)  z^{n-1}
\label{ieq17}%
\end{equation}
and%
\begin{equation}
(1-\lambda)\left(  \frac{g(w)}{w}\right)  ^{\mu}+\lambda g^{\prime}(w)\left(
\frac{g(w)}{w}\right)  ^{\mu-1}+\xi\delta wg^{\prime\prime}(w)=1+\sum
\limits_{n=2}^{\infty}F_{n-1}\left(  A_{2},A_{3},...,A_{n}\right)  w^{n-1}
\label{ieq18}%
\end{equation}
where%
\begin{equation}
F_{1}=(\mu+\lambda+2\xi\delta)a_{2},\text{ }F_{2}=(\mu+2\lambda+6\xi
\delta)\left[  \frac{\mu-1}{2}a_{2}^{2}+\left(  1+\frac{6\delta}{2\lambda
+1}\right)  a_{3}\right]  \label{ieq20}%
\end{equation}
and, in general (see \cite{faber})%
\begin{align}
F_{n-1}\left(  a_{2},a_{3},...,a_{n}\right)   &  =\left[  \mu+\left(
n-1\right)  \lambda+n\left(  n-1\right)  \xi\delta\right]  \times\left[
\left(  \mu-1\right)  !\right] \label{ieq21}\\
&  \ \ \times\left[  \sum\limits_{n=2}^{\infty}\frac{a_{2}^{i_{1}}a_{3}%
^{i_{2}}...a_{n}^{i_{n-1}}}{i_{1}!i_{2}!\cdots i_{n}!\left[  \mu-\left(
i_{1}+i_{2}+\cdots+i_{n-1}\right)  \right]  !}\right]
\end{align}

Next, by using the Faber polynomial expansion of functions $\varpi,\varphi
\in\mathcal{C}$, we also obtain%
\begin{equation}
\Phi\left(  \varpi(z)\right)  =1+\sum\limits_{n=1}^{\infty}\sum\limits_{k=1}%
^{\infty}\Phi_{n}F_{n}^{k}\left(  t_{1},t_{2},...,t_{n}\right)  z^{n},
\label{ieq22}%
\end{equation}
and%
\begin{equation}
\Phi\left(  \varphi(z)\right)  =1+\sum\limits_{n=1}^{\infty}\sum
\limits_{k=1}^{\infty}\Phi_{n}F_{n}^{k}\left(  s_{1},s_{2},...,s_{n}\right)
w^{n}. \label{ieq25}%
\end{equation}

Comparing the corresponding coefficients yields
\[
\left[  \mu+\left(  n-1\right)  \lambda+n\left(  n-1\right)  \xi\delta\right]
a_{n}=\sum\limits_{k=1}^{n-1}\Phi_{n}F_{n-1}^{k}\left(  t_{1},t_{2}%
,...,t_{n-1}\right)  \text{ }(n\geq2)
\]
and%
\begin{equation}
\left[  \mu+\left(  n-1\right)  \lambda+n\left(  n-1\right)  \xi\delta\right]
A_{n}=\sum\limits_{k=1}^{n-1}\Phi_{n}F_{n-1}^{k}\left(  s_{1},s_{2}%
,...,s_{n-1}\right)  \text{ }(n\geq2). \label{ieq26}%
\end{equation}

Note that for $a_{m}=0,$ $2\leq m\leq n-1,$ we have $A_{n}=-a_{n}$ and so
\[
\left[  \mu+\left(  n-1\right)  \lambda+n\left(  n-1\right)  \xi\delta\right]
a_{n}=\Phi_{1}t_{n-1},
\]%
\begin{equation}
-\left[  \mu+\left(  n-1\right)  \lambda+n\left(  n-1\right)  \xi
\delta\right]  a_{n}=\Phi_{1}s_{n-1}, \label{ieq27}%
\end{equation}

Now taking the absolute values of either of the above two equations and using
the facts that $\left\vert \Phi_{1}\right\vert \leq2,$ $\left\vert
t_{n-1}\right\vert \leq1,$ and $\left\vert s_{n-1}\right\vert \leq1,$ we
obtain
\begin{align}
\left\vert a_{n}\right\vert  &  \leq\frac{\left\vert \Phi_{1}t_{n-1}%
\right\vert }{\mu+\left(  n-1\right)  \lambda+n\left(  n-1\right)  \xi\delta
}=\frac{\left\vert \Phi_{1}s_{n-1}\right\vert }{\mu+\left(  n-1\right)
\lambda+n\left(  n-1\right)  \xi\delta}\label{ieq28}\\
&  \leq\frac{2}{\mu+\left(  n-1\right)  \lambda+n\left(  n-1\right)  \xi
\delta}%
\end{align}

This evidently completes the proof of \ref{thm2}.
\end{proof}

\begin{remark}
\label{rem2} As a final remark, for $\delta=0$ in

(i) Theorem \ref{thm1} we obtain Theorem 1 in \cite{atl}. \vspace{0.05in}

(ii) Theorem \ref{thm2} we obtain Theorem 2 in \cite{atl}. \vspace{0.05in}
\end{remark}

\end{document}